\documentclass[12pt]{amsart}

\usepackage{cancel}
\usepackage{latexsym, amssymb, amsmath,esint}
\usepackage{soul}
\usepackage{amsfonts, graphicx}
\usepackage{graphicx,color}
\newcommand{\be}{\begin{equation}}
\newcommand{\ee}{\end{equation}}
\newcommand{\beq}{\begin{eqnarray}}
\newcommand{\eeq}{\end{eqnarray}}

\usepackage{wasysym,stmaryrd}
\newtheorem{thm}{Theorem}[section]

\newtheorem{lma}[thm]{Lemma}
\newtheorem{prop}[thm]{Proposition}

\newtheorem{defn}[thm]{Definition}

\newtheorem{rem}[thm]{Remark}
\numberwithin{equation}{section}

\newtheorem*{GRH}{Generic Regularity Hypothesis}

\def\be{\begin{equation}}
\def\ee{\end{equation}}
\def\bee{\begin{equation*}}
\def\eee{\end{equation*}}

\def\Ric{\text{\rm Ric}}
\def\Rm{\text{\rm Rm}}

\def\ve{\varepsilon}

\def\vp{\varphi}

\newcommand{\ti}[1]{\tilde{#1}}

\DeclareMathOperator{\sys}{sys}
\DeclareMathOperator{\Area}{Area}
\DeclareMathOperator{\bRic}{biRic}
\DeclareMathOperator{\dist}{dist}
\DeclareMathOperator{\Div}{div}

\begin{document}

\title[]
{Homological $n$-systole in $(n+1)$-manifolds and bi-Ricci curvature}

\author{Jianchun Chu}
\address[Jianchun Chu]{School of Mathematical Sciences, Peking University, Yiheyuan Road 5, Beijing 100871, People's Republic of China}
\email{jianchunchu@math.pku.edu.cn}

\author{Man-Chun Lee}
\address[Man-Chun Lee]{Department of Mathematics, The Chinese University of Hong Kong, Shatin, Hong Kong, China}
\email{mclee@math.cuhk.edu.hk}

\author{Jintian Zhu}
\address[Jintian Zhu]{Institute for Theoretical Sciences, Westlake University, 600 Dunyu Road, 310030, Hangzhou, Zhejiang, People's Republic of China}
\email{zhujintian@westlake.edu.cn}

\renewcommand{\subjclassname}{\textup{2020} Mathematics Subject Classification}

\date{\today}

\begin{abstract}
In this paper, we prove an optimal systolic inequality and the corresponding rigidity in the equality case on closed manifolds with positive bi-Ricci curvature, which generalizes the work of Bray-Brendle-Neves in \cite{BrayBrendleNeves2010}. The proof is given in all dimensions based on the method of minimal surfaces under the Generic Regularity Hypothesis, which is known to be true up to dimension ten.
\end{abstract}

\maketitle

\markboth{Jianchun Chu, Man-Chun Lee and Jintian Zhu}{Homological $n$-systole in $(n+1)$-manifolds and bi-Ricci curvature}

\section{Introduction}

Let $(M^{n+1},g)$ be an oriented closed Riemannian manifold. When $M^2$ is two dimensional with uniformly positive Gaussian curvature $K$, the classical theorem of Toponogov \cite{Toponogov1959} states that every simple closed geodesic $\gamma$ has length bounded from above by $2\pi\cdot (\inf_M {K})^{-1/2}$ and the corresponding rigidity holds on $M=\mathbb{S}^2$.  This raised an interesting restriction phenomenon relating the positivity of curvature and ``minimal" sub-manifolds. It is natural to look for a generalization in higher dimension. However, even in the standard $3$-sphere the area of a closed embedded minimal surface can be arbitrarily large (refer to \cite[Theorem B]{Zhou2020}). This means that a direct generalization of the Toponogov theorem to dimension three is out of reach and a modification is necessary in general.

In the presence of positive scalar curvature, partially motivated by the work of Schoen-Yau \cite{SchoenYau1979} on the resolution of Geroch's conjecture in dimension three, Bray-Brendle-Neves \cite{BrayBrendleNeves2010} considered three-manifolds $M^3$ with scalar curvature bounded from below by 2 with a locally area-minimizing embedded 2-sphere $\Sigma\subset M$. Particularly, they show that $\Sigma^2$ has area less than or equal to $4\pi$. More importantly, the equality holds only if $M^3$ is covered by $\mathbb{S}^2\times \mathbb{R}$. The infimum of all such embedded 2-sphere are usually referred as spherical $2$-systole in dimension three. Since then, there has been many interesting development in understanding stable minimal objects in manifolds with positive scalar curvature. For instances, a similar inequality for embedded $\mathbb{RP}^2$ is also studied by Bray-Brendle-Eichmair-Neves \cite{BrayBrendleEichmairNeves2010}, see also the notion of width studied by Colding-Minicozzi \cite{ColdingMinicozzi2005}. In the non-compact case, the third-named author \cite{Zhu2023} considered the splitting theorem of Bray-Brendle-Neves \cite{BrayBrendleNeves2010}.

In this work, we are interested in objects which are minimal within the homology group. For $(M^{n+1}, g)$ with non-trivial \textit{homology} group  $H_n(M)$, we define the {\it homological $n$-systole} to be
$$
\sys_n(M,g):=\inf\left\{\Area_g(\Sigma)
\left|\begin{array}{c}\Sigma \mbox{ is a smooth hypersurface}\\
\mbox{ in } M\mbox{ with }[\Sigma]\neq 0\in H_n(M)
\end{array}\right.\right\}.
$$
When $n=2$, it was also considered by Stern \cite{Stern2022} using the level set approach. It is also dominated by the spherical $2$-systole when $n=2$, and thus it follows from  \cite{BrayBrendleNeves2010} that in dimension 3,
\begin{equation*}
    \label{intro:eqn}\sys_2(M,g)\cdot \min_M R(g) \leq 8\pi,
\end{equation*}
where the equality holds only if $M^3$ is covered by cylinder. We refer  inequalities of this type as systolic inequality.

Motivated by this, it is tempting to ask if we can estimate (optimally) the left-hand side in \eqref{intro:eqn} from above in all dimensions. However without topological control, it is hard to expect any quantitative control of area under only scalar curvature lower bound. One way to circumvent this is to strengthen the curvature. In this work, we are interested in generalizing the systolic inequality to higher dimension under the positive bi-Ricci curvature. This curvature condition was first introduced by Shen-Ye \cite{ShenYe1996} concerning its interplay with the geometry of stable minimal hypersurfaces. Given a Riemannian manifold $(M,g)$,  for each pair of vectors $(v,w)$ we define the bi-Ricci curvature as
$$
\bRic(v,w)=\Ric(v,v)+\Ric(w,w)-\Rm(v,w,w,v).
$$
When the dimension is three, this coincides with one half of the scalar curvature for any pair of $(v,w)$ which is orthonormal while positive bi-Ricci curvature is stronger in higher dimension and is a special case of the intermediate curvature introduced by Brendle-Hirsch-Johne \cite{BrendleHirschJohne2024}. This curvature notion appears naturally in the study of weighted minimal slicings \cite{BrendleHirschJohne2024} or $\mu$-bubbles \cite{Xu2023}. Perhaps most surprisingly, it played an important role in the resolution of the stable Bernstein problem \cite{ChodoshLiMinterStryker2024, Mazet2024}.

Motivated by the above mentioned works and development, the goal of this paper is to investigate the interplay between homological $n$-systole of a {\it closed} Riemannian manifold and its bi-Ricci curvature. There have been many researches  on the geometric effects of bi-Ricci curvature with dimension constraints, for example see \cite{ShenYe1996,Xu2023,BrendleHirschJohne2024,AntonelliXu2024}. A feature of our work here is that we shall handle the $n$-systole problem in {\it all} dimensions using the method of minimal surfaces under the following {\it Generic Regularity Hypothesis}.

\begin{GRH}
    Let $(M^{n+1},g)$ be a closed, oriented $(n+1)$-dimensional Riemannian manifold. There exist arbitrarily $C^\infty$-small perturbations $g'$ of $g$ with the property that for every non-zero homology class $\alpha\in H_n(M)$ there exist disjoint, smooth, closed, oriented hypersurfaces $M_1,\ldots, M_Q$ and intergers $k_1,\ldots, k_Q$ so that $\sum_{i=1}^Qk_iM_i$ is of least area in the homology class $\alpha$ with respect to $g'$.
\end{GRH}
The validity of the Generic Regularity Hypothesis is one of the important conjectures in geometric measure theory.  Recently, the Generic Regularity Hypothesis has been verified to be true up to dimension $n+1=10$. For this topic we refer interested readers to \cite{Smale1993,ChodoshLiokumovichSpolaor2020,LiWang2020,ChodoshMantoulidisSchulze2023} and the references therein.

For any continuous function $\sigma:M\to \mathbb R$, we say that the bi-Ricci curvature is bounded from below by $\sigma$ if $\bRic(v,w)\geq \sigma(p)$ at each point $p$ for any orthonormal pair $(v,w)$ in $T_pM$. In this case, we abbreviate it as  $\bRic \geq \sigma$ for short. Our main theorem is stated as follows:

\begin{thm}\label{Thm: main}
Let $(M^{n+1},g)$ be an orientable $(n+1)$-dimensional closed Riemannian manifold with $H_n(M)\neq 0$ and $\bRic\geq n-1$. Suppose additionally that the Generic Regularity Hypothesis holds.
Then we have
\[
\sys_n(M^{n+1},g)\leq|\mathbb S^n|,
\]
where the equality yields that the universal cover of $(M,g)$ splits isometrically as $\mathbb S^n\times \mathbb R$.
\end{thm}

This is largely motivated by the recent work of Antonelli-Xu \cite{AntonelliXu2024} in estimating the volume growth of manifolds with spectral Ricci lower bound, see also \cite{ChodoshLiMinterStryker2024}. The main challenge lies in studying its rigidity due to the presence of singularity on the minimizer. To overcome this, we make use of the metric deformation trick of Liu  \cite{Liu2013} and Chodosh-Eichmair-Moraru \cite{ChodoshEichmairMoraru2019}, to produce some foliation of minimizer and improve the curvature non-negativity in the presence of minimizer under the Generic Regularity Hypothesis.

\bigskip

The paper will be arranged as follows. In Section \ref{Sec: low dim}, we prove Theorem \ref{Thm: main} with the additional dimension assumption $n+1\leq 7$, where area-minimizing hypersurfaces are automatically smooth and so we don't really use the Generic Regularity Hypothesis. In Section \ref{Sec: generic regularity}, we deal with the case when $n+1\geq 8$. In Appendix \ref{Appendix}, we collect some known curvature formulas under conformal deformation.

\bigskip

{\bf Acknowledgements.} The first-named author was partially supported by National Key R\&D Program of China 2023YFA1009900 and NSFC grant 12271008. The second-named author was partially supported by Hong Kong RGC grant (Early Career Scheme) of Hong Kong No. 24304222, No. 14300623, and a NSFC grant No. 12222122. The third-named author was partially supported by National Key R\&D Program of China 2023YFA1009900 as well as the startup fund from Westlake University.

\section{Proof of Theorem \ref{Thm: main} when $n+1\leq 7$}\label{Sec: low dim}
\subsection{Systolic inequality}
We first recall the notion of spectral Ricci curvature:
\begin{defn}
Let $(N,h)$ be a complete Riemannian manifold. Given any constant $\gamma>0$ we say that $(N,h)$ has $\gamma$-spectral Ricci curvature bounded from below by a constant $\mu$ if there is a smooth positive function $u$ such that
\begin{equation}\label{Eq: spectral function}
-\gamma\Delta_h u+\lambda_\Ric u\geq \mu u,
\end{equation}
where $\lambda_\Ric$ is the continuous function on $N$
 given by
 $$
 \lambda_\Ric(x) = \min_{v\in T_{x}N, \, |v|=1}\Ric(v,v).
 $$
 \end{defn}

 Our proof is based on the following volume comparison theorem proved by Antonelli-Xu \cite{AntonelliXu2024} recently.
 \begin{thm}[Theorem 1 of \cite{AntonelliXu2024}]\label{Thm: spectral comparison}
 For any constant
 $$\gamma\in \left[0,\frac{n-1}{n-2}\right]$$ if $(N,h)$ is a closed Riemannian $n$-manifold with $n\geq 3$ having $\gamma$-spectral Ricci curvature bounded from below by $n-1$, then we have $$\Area(N)\leq |\mathbb S^n|$$ and the equality holds if and only if $N$ is isometric to the standard unit $n$-sphere $\mathbb S^n$.
 \end{thm}
 \begin{rem}\label{Rem: integral estimate}
 From the proof of above theorem in \cite{AntonelliXu2024} one actually has the stronger conclusion
 \begin{equation}\label{Eq: integral of u small p}
 \int_Nu^{\frac{2\gamma}{n-1}}\,\mathrm d\sigma\leq |\mathbb S^n|,
 \end{equation}
 where $u$ is the positive smooth function in \eqref{Eq: spectral function} with normalization $\min_N u=1$. This will be used to prove the rigidity part of Theorem \ref{Thm: main} in Section \ref{Sec: generic regularity}.
 \end{rem}

 \begin{lma}\label{Lem: infinitesimal rigidity}
 Suppose that $(M^{n+1},g)$ is a Riemannian manifold with $\bRic\geq n-1$ and that $N^n$ is a smooth, connected, stable, closed and two-sided minimal hypersurface in $(M,g)$. Then we have $\Area(N)\leq |\mathbb S^n|$, where the equality implies
 \begin{itemize}\setlength{\itemsep}{1mm}
 \item $N$ is isometric to the standard unit $n$-sphere $\mathbb S^n$;
 \item the normal Ricci curvature $\Ric_M(\nu,\nu)$ vanishes along $N$;
 \item $N$ is totally geodesic.
 \end{itemize}
 \end{lma}
 \begin{proof}
 Since $N$ is closed and two-sided, the stability of $N$ yields
 \begin{equation}\label{Eq: stability}
 \int_N|\nabla \vp|^2-\Big(\Ric_M(\nu,\nu)+\|A\|^2\Big)\vp^2\,\mathrm d\sigma\geq 0
 \end{equation}
 for any test function $\varphi\in C^\infty(N)$.  Take $u$ to be the first eigenfunction of the Jacobi operator
 $$\mathcal J=-\Delta_N-\Big(\Ric_M(\nu,\nu)+\|A\|^2\Big).$$ We are going to verify
 \begin{equation}\label{Eq: spectral function-2}
 -\Delta_Nu+\lambda_{\Ric,N}u\geq (n-1)u.
 \end{equation}
 In other words, $N$ has $1$-spectral Ricci curvature bounded from below by $n-1$.

Since the function $u$ is positive and the first eigenvalue is nonnegative from the stability inequality \eqref{Eq: stability}, we have
\[
-\Delta_Nu\geq \Big(\Ric_M(\nu,\nu)+\|A\|^2\Big)u.
\]
Fix an arbitrary point $p$ in $N$ and then take an orthonormal frame $\{e_i\}_{i=1}^n$ of $T_pN$ such that $\lambda_{\Ric,N}(p)=\Ric_N(e_1,e_1)$. From the Gauss equation we can compute at the point $p$ that
$$
\Ric_M(\nu,\nu) = \bRic_M(\nu,e_1)-\lambda_{\Ric,N}(p)+\sum_{j=2}^{n}\left(A_{11}A_{jj}-A_{1j}^2\right).
$$
Since $N$ is minimal, we see that
\begin{equation*}
\sum_{j=2}^{n}\left(A_{11}A_{jj}-A_{1j}^2\right) = -A_{11}^{2}-\sum_{j=2}^{n}A_{1j}^2 = -\sum_{j=1}^{n}A_{1j}^2.
\end{equation*}
Therefore, at the point $p$ we have
\begin{equation}\label{Eq: spectral function-1}
\begin{split}
-\Delta_Nu \geq {} & \Big(\Ric_M(\nu,\nu)+\|A\|^2\Big)u \\[1mm]
= {} & \Big(\bRic_M(\nu,e_1)-\lambda_{\Ric,N}(p)+\|A\|^{2}-\sum_{j=1}^{n}A_{1j}^2\Big)u \\
\geq {} & \Big((n-1)-\lambda_{\Ric,N}\Big)u,
\end{split}
\end{equation}
which gives the desired inequality \eqref{Eq: spectral function-2}.
 It follows from Theorem \ref{Thm: spectral comparison} that $$\Area(N)\leq |\mathbb S^n|.$$

\smallskip

Next we discuss the equality case. From the rigidity part of Theorem \ref{Thm: spectral comparison} we know that $N$ is isometric to the standard unit $n$-sphere $\mathbb S^n$. In particular, we must have $\lambda_{\Ric,N}\equiv n-1$ and so \eqref{Eq: spectral function-1} implies that $u$ is a positive constant on $N$. The equality of \eqref{Eq: spectral function-1} yields $\bRic_M(\nu,e_1)=n-1$ and $\|A\|^2=\sum_{j=1}^n A_{1j}^2$ along $N$. Combined with the minimality of $N$ we see that $N$ is totally geodesic and that the normal Ricci curvature $\Ric_M(\nu,\nu)$ vanishes along $N$.
\end{proof}

\begin{prop}\label{Prop: systole inequality}
Let $(M^{n+1},g)$ be an orientable $(n+1)$-dimensional closed Riemannian manifold with $H_n(M)\neq 0$ and $\bRic\geq n-1$. If $n+1\leq 7$, then we can find a connected, smooth, embedded, homologically non-trivial hypersurface $N_{\mathrm{sys}}$ such that
\[
\sys_n(M,g)=\Area(N_{\mathrm{sys}})\leq|\mathbb S^n|.
\]
If $\sys_n(M,g)=|\mathbb S^n|$, then
\begin{itemize}\setlength{\itemsep}{1mm}
\item $N_{\mathrm{sys}}$ is isometric to the standard unit $n$-sphere $\mathbb S^n$;
\item the normal Ricci curvature $\Ric_M(\nu,\nu)$ vanishes along $N_{\mathrm{sys}}$;
\item $N_{\mathrm{sys}}$ is totally geodesic.
\end{itemize}
\end{prop}
 \begin{proof}
It comes from the geometric measure theory that we can find a smooth embedded oriented minimal hypersurface $N_{\mathrm{sys}}$ with integer multiplicity such that
\begin{itemize}
\item $N_{\mathrm{sys}}$ represents a non-trivial homology class in $H_n(M)$ and satisfies
$$\Area(N_{\mathrm{sys}})=\sys_n(M,g).$$
\item $N_{\mathrm{sys}}$ satisfies the following holomogically area-minimizing property: there holds
$$\Area(N_{\mathrm{sys}})\leq\Area(N)$$ for any smooth oriented hypersurface $N$ representing the same homology class as $[N_{\mathrm{sys}}]$.
\end{itemize}
For our purpose, we write
$$
N_{\mathrm{sys}}=\sum_{i=1}^k m_iN_i,
$$
where each  $N_i$ is an  oriented connected smooth embedded hypersurface and $m_i$ is the integer multiplicity of $N_i$.  Since $\Area(N_{\mathrm{sys}})=\sys_n(M,g)$, we conclude that all hypersurfaces $N_i$ are homologically non-trivial, and so $N_{\mathrm{sys}}$ has to be a connected, smooth, embedded oriented hypersurface with multiplicity one. The homologically area-minimizing property of $N_{\mathrm{sys}}$ implies the stability of $N_{\mathrm{sys}}$ as a minimal hypersurface and so Lemma \ref{Lem: infinitesimal rigidity} yields
$$\Area(N_{\mathrm{sys}})\leq |\mathbb S^n|.$$
In the case when $\sys_n(M,g)=|\mathbb S^n|$, we have $\Area(N_{\mathrm{sys}})=|\mathbb S^n|$. The rigidity part follows from Lemma \ref{Lem: infinitesimal rigidity}.
\end{proof}

\subsection{Rigidity}
\begin{prop}\label{Prop: totally geodesic approximation}
Under the assumptions of Proposition \ref{Prop: systole inequality}, let $N_{\mathrm{sys}}$ be the hypersurface obtained in Proposition \ref{Prop: systole inequality}. When $\Area(N_{\mathrm{sys}})=|\mathbb S^n|$, there is a sequence of totally geodesic hypersurfaces $N_i$ converging to $N_{\mathrm{sys}}$ smoothly from one side of $N_{\mathrm{sys}}$.
\end{prop}

To prove this proposition we shall use the {\it metric-deformation trick} inspired by the work of Liu \cite{Liu2013} proving the Milnor conjecture in dimension three and also the work of Chodosh-Eichmair-Moraru \cite{ChodoshEichmairMoraru2019} concerning a rigidity result for complete $3$-manifolds with nonnegative scalar curvature in the presence of minimizing cylinders.

\subsubsection{Preliminaries}
By the compactness of $M$, we let $r_{\mathrm{inj}}$ be the injective radius of $(M,g)$. Given any point $p\in M$ and any constants $r\in (0,r_{\mathrm{inj}})$ and $t\in(0,1)$ we consider the following conformal deformation of $g$:
$$
g_{p,r,t}=e^{-2t(r^2-\rho^2)^5}g,$$
where
$$
\rho=\min\{r,\dist_g(\cdot,p)\}.
$$
It is straightforward to verify that each metric $g_{p,r,t}$ is a $C^{4,\alpha}$-metric on $M$ and that the metrics $g_{p,r,t}$ converge to $g$ in $C^{4,\alpha}$-sense as $t\to 0^+$ when $p$ and $r$ are fixed. Moreover, we have $g_{p,r,t}<g$ in the geodesic ball $B_r(p)$.

\begin{lma}\label{Lem: Hessian bound}
There are constants $0<r_0<r_{\mathrm{inj}}$ and $C_0>0$ such that if $r<r_0$ then we have
\begin{equation*}
\rho|\nabla^2\rho|_g\leq C_0\mbox{ when }\rho<r.
\end{equation*}
\end{lma}
\begin{proof}
Notice that the Hessian $\nabla^2\rho$ is just the second fundamental form $A$ of the geodesic spheres when $\rho<r$. From the Ricatti equation we have
$$
\left(\frac{\mathrm d}{\mathrm ds}A\right)(\cdot,\cdot)=-\Rm(\cdot,\gamma',\gamma',\cdot)-A^2(\cdot,\cdot)
$$
along any minimizing geodesic $\gamma=\gamma(s):[0,r_{\mathrm{inj}})\to (M,g)$ starting from the point $p$. Since $(M,g)$ is closed, we can assume $\|\Rm\|\leq \Lambda$ for some universal constant $\Lambda$ independent of $p$ and $r$. Denote $\lambda_{\mathrm{min}}$ and $\lambda_{\mathrm{max}}$ to be the minimal and maximal eigenvalues of $A$ respectively. It is clear that $\lambda_{\mathrm{min}}$ and $\lambda_{\mathrm{max}}$ are locally Lipschitz functions satisfying
$$
\lambda_{\mathrm{min}}'\geq -\Lambda-\lambda_{\mathrm{min}}^2\mbox{ and }\lambda_{\mathrm{max}}'\leq \Lambda-\lambda_{\mathrm{max}}^2.
$$
Using the fact $\lambda_{\mathrm{min}}\sim \lambda_{\mathrm{max}}\sim s^{-1}$ as $s\to 0$ we can derive
$$
\sqrt\Lambda\cdot \frac{\cos{\sqrt\Lambda s}}{\sin{\sqrt\Lambda s}}\leq \lambda_{\mathrm{min}}\leq \lambda_{\mathrm{max}}\leq \sqrt\Lambda\cdot \frac{\cosh(\sqrt\Lambda s)}{\sinh(\sqrt\Lambda s)}
$$
when $0<s<\pi/\sqrt\Lambda$. By taking $r_0=\pi/\sqrt{2\Lambda}$ we can determine the constant $C_0$ correspondingly from above inequality and this completes the proof.
\end{proof}
In the following, we always assume $r<r_0$ such that Lemma \ref{Lem: Hessian bound} holds.
\begin{lma}\label{Lem: biRicci increase}
There exists a constant $\theta\in (0,1)$ independent of $p$, $r$ and $t$ such that it holds in the annulus region $\mathcal A_{\theta,r}(p):=B_r(p)\setminus \overline{B_{\theta r}(p)}$ that
$$
e^{-nt(r^2-\rho^2)^5}\cdot \bRic_{g_{p,r,t}}(v_{p,r,t},w_{p,r,t})\geq  \bRic_g(v,w)+100tr^2(r^2-\rho^2)^3,
$$
where $(v,w)$ is any $g$-orthonormal pair and $(v_{p,r,t},w_{p,r,t})$ is the $g_{p,r,t}$-unit normalization of $(v,w)$.
\end{lma}
\begin{proof}
This follows from a straightforward computation. Denote
$$
f=t(r^2-\rho^2)^5.
$$
Then we have the gradient bound
$$
|\nabla_gf|_g\leq Ct\rho(r^2-\rho^2)^4,
$$
and the Hessian estimate
$$
\nabla^2_g f \geq 80t\rho^2(r^2-\rho^2)^3\mathrm d\rho\otimes \mathrm d\rho-Ct(r^2-\rho^2)^4(\mathrm d\rho\otimes \mathrm d\rho+\rho|\nabla^2_g\rho|_g).
$$
Here and in the sequel, we denote $C$ to be an absolute constant independent of $p$, $r$ and $t$, which may be different from line to line. It follows from Lemma \ref{Lem: Hessian bound} that we have
$
\rho|\nabla^2_g\rho|_g\leq C.
$
Using formula \eqref{Eq: biRicci change} we conclude that for any $g$-orthonormal pair $(v,w)$ it holds
\[
\begin{split}
&e^{-nt(r^2-\rho^2)^5}\cdot \bRic_{g_{p,r,t}}(v_{p,r,t},w_{p,r,t})\\
&\qquad\geq e^{-(n-2)t(r^2-\rho^2)^5}\Big (\bRic_g(v,w)+160t\rho^2(r^2-\rho^2)^3-Ct(r^2-\rho^2)^4\Big).
\end{split}
\]
Combining this with $1-Ct(r^{2}-\rho^{2})^{5}\leq e^{-(n-2)t(r^2-\rho^2)^5}\leq1$, we obtain
\[
\begin{split}
&e^{-nt(r^2-\rho^2)^5}\cdot \bRic_{g_{p,r,t}}(v_{p,r,t},w_{p,r,t})\\
&\qquad\geq \bRic_g(v,w)+160t\rho^2(r^2-\rho^2)^3-Ct(r^2-\rho^2)^4.
\end{split}
\]
Take $\theta$ sufficiently close to $1$ such that
$$
C(1-\theta^2)<40\theta^2\mbox{ and }120\theta^2>100.
$$
Then we can guarantee
$$
e^{-nt(r^2-\rho^2)^5}\cdot \bRic_{g_{p,r,t}}(v_{p,r,t},w_{p,r,t})\geq  \bRic_g(v,w)+100tr^2(r^2-\rho^2)^3
$$
in $\mathcal A_{\theta,r}(p)$.
\end{proof}

\begin{lma}\label{Lem: Ricci increase}
Given a positive constant $\Lambda_{\mathrm{cur}}$, there exist constants $\iota>0$ and $\theta\in(0,1)$ such that if $N$ is a smooth embedded hypersurface satisfying $\|A\|_g\leq \Lambda_{\mathrm{cur}}$ and $|\nabla_{N,g}\rho|_g\leq \iota$ in $B_r(p)$, then we have
$$
e^{-nt(r^2-\rho^2)^5}\cdot \Ric^N_{g_{p,r,t}}(v_{p,r,t},v_{p,r,t})\leq \Ric_g^N(v,v)+50tr^2(r^2-\rho^2)^3
$$
in $\mathcal A_{\theta,r}(p)\cap N$,
where $v$ is any $g$-unit vector tangential to $N$ and $v_{p,r,t}$ is the $g_{p,r,t}$-unit normalization of $v$.
\end{lma}
\begin{proof}
Since $M$ is closed, we have $\|\Rm^M_g\|_g\leq \Lambda$ for some universal constant $\Lambda$. From the Gauss equation and $\|A\|_g\leq \Lambda_{\mathrm{cur}}$  we know that $\|\Ric_g^N\|\leq C$. Through a direct computation we have
$$
|\nabla_{N,g}f|_g\leq Ct\iota\rho(r^2-\rho^2)^4
$$
for $r>\rho$,
and
$$
|\nabla_{N,g}^2 f|_g\leq 80t\iota^2\rho^2(r^2-\rho^2)^3+Ct\iota^2(r^2-\rho^2)^4+Ct(r^2-\rho^2)^4\rho|\nabla^2_{N,g}\rho|_g.
$$
Using the decomposition
$$
\nabla^2_{N,g}\rho=\nabla_g^2\rho|_N-\langle\nabla_{g}\rho,\nu\rangle A
$$
we obtain
$$\rho|\nabla^2_{N,g}\rho|_g \leq C.$$
The formula \eqref{Eq: Ricci change} then yields
\[
\begin{split}
&e^{-nt(r^2-\rho^2)^5}\cdot\Ric^N_{g_{p,r,t}}(v_{p,r,t},v_{p,r,t})\\
&\qquad\leq e^{-(n-2)t(r^2-\rho^2)^5}\Big (\Ric_g^N(v,v)+Ct\iota^2r^2(r^2-\rho^2)^3+Ct(r^2-\rho^2)^4\Big)\\
&\qquad\leq\Ric_g^N(v,v)+Ct\iota^2r^2(r^2-\rho^2)^3+Ct(r^2-\rho^2)^4.
\end{split}
\]
By taking $\theta$ sufficiently close to $1$ and $\iota$ small enough we can guarantee
$$
Ct\iota^2r^2(r^2-\rho^2)^3+Ct(r^2-\rho^2)^4\leq 50tr^2(r^2-\rho^2)^3.
$$
This completes the proof.
\end{proof}

\subsubsection{Proof of Proposition \ref{Prop: totally geodesic approximation}}\label{Proposition 2.6}
In the following, we are going to construct totally geodesic hypersurfaces arbitrarily close to $N_{\mathrm{sys}}$ on one side. In order to guarantee the one-sideness, we cut the manifold $M$ along the hypersurface $N_{\mathrm{sys}}$. Namely, we consider the metric completion $(\hat M,\hat g)$ of the complement $(M\setminus N_{\mathrm{sys}},g)$.

Recall that $N_{\mathrm{sys}}$ is homologically non-trivial. This implies that $N_{\mathrm{sys}}$ is non-separating in $M$ and so $\partial \hat M$ consists of two copies of $N_{\mathrm{sys}}$. From now on, we fix one component of $\partial \hat M$ and denote it by $\hat N$. Our goal is to construct totally geodesic hypersurfaces arbitrarily close to $\hat N$ in $(\hat M,\hat g)$.

Take a minimizing geodesic
$\hat\gamma:[0,\ve)\to (\hat M,\hat g)$
with $\hat\gamma(0)\in \hat N$ orthogonal to $\hat N$ at $\hat\gamma(0)$. By taking $\ve$ small enough we can guarantee
$$\dist_{\hat g}(\hat\gamma(s),\hat N)=s \mbox{ for all }s\in[0,\ve)$$
as well as $2\ve<r_{0}$, where $r_0$ is the constant from Lemma \ref{Lem: Hessian bound}. In particular, we have
\begin{itemize}\setlength{\itemsep}{1mm}
\item the distance function $\dist_{\hat g}(\cdot,\hat \gamma(s))$ is smooth in $\hat B_{2s}(\hat \gamma(s))$;
\item $\hat\gamma(0)$ is the unique nearest point of  $\hat N$ to $\hat\gamma(s)$ for all $s\in[0,\ve)$.
\end{itemize}

Let us work with a fixed $s\in(0,\ve)$ and $\hat p=\hat\gamma(s)$. Denote
$$\hat d=\dist_{\hat g}(\cdot,\hat p)$$ and $\hat \nu$ to be the outward unit normal vector field of $\hat N$ in $(\hat M,\hat g)$.
From our choice of $\hat\gamma$ we have $\langle \nabla_{\hat g} \hat d,\hat \nu\rangle=1$ at the point $\hat\gamma(0)$.

Fix $\theta_0\in(0,1)$ and $\iota_0\in (0,1)$ such that Lemma \ref{Lem: biRicci increase} holds and that Lemma \ref{Lem: Ricci increase} holds with the choice
\begin{equation}\label{Eq: Lambda cur}
\Lambda_{\mathrm{cur}}=\max\|A_{\hat N}\|_{\hat g}+1.
\end{equation}
It follows from the smoothness of $\hat N$ and $\hat d$ that we can guarantee $\langle \nabla_{\hat g} \hat d,\hat \nu\rangle>\sqrt{1-\iota_0^2}$ in an open neighborhood $\mathcal N$ of $\hat\gamma(0)$ on $\hat N$. Recall that $\hat\gamma(0)$ is the unique nearest point of $\hat p$ on $\hat N$. This yields $\dist_{\hat g}(\hat N\setminus \mathcal N,\hat p)>s$ and so there is a constant $\delta\in(0,s)$ such that
\begin{equation}\label{Eq: star shape}
    \langle \nabla_{\hat g} \hat d,\hat \nu\rangle> \sqrt{1-\iota_0^2}\mbox{ on }\hat B_{s+\delta}(\hat p)\cap\hat N.
\end{equation}

Let $r=s+\tau$ with $0<\tau<\delta$ to be determined later. In the following, we consider the conformal metric
$$
\hat g_{r,t}=e^{-2t(r^2-\hat\rho^2)^5}\hat g
$$
where
$$
\hat\rho=\min\{r,\dist_{\hat g}(\cdot,\hat p)\}.
$$
Denote $$\phi:\hat M\to M$$
to be the natural gluing map.
It is easy to verify that when $s$ is small enough the metric $\hat g_{r,t}$ coincides with $\phi^*g_{\phi(\hat p),r,t}$ in $\hat B_r(\hat p)$, where $\hat B_r(\hat p)$ is denoted to be the geodesic $r$-ball centered at point $\hat p$ in $(\hat M,\hat g)$.

Now we are in a position to prove Proposition \ref{Prop: totally geodesic approximation}.

\begin{proof}[Proof of Proposition \ref{Prop: totally geodesic approximation}]
Denote $\mathcal C$ to be the collection of hypersurfaces which enclose a region with $\hat N$. Since $N_{\mathrm{sys}}$ is minimal in $(M,g)$, $\partial \hat M$ is minimal in $(\hat M,\hat g)$ as well. Then it follows from a direct computation and \eqref{Eq: star shape} that the mean curvature $\hat H_{r,t}$ of $\hat N$ in $(\hat M,\hat g_{r,t})$ with respect to the outward unit normal vector field $\hat\nu$ satisfies
$$
\hat H_{r,t}=10ne^{t(r^2-\hat\rho^2)^5}t(r^2-\hat\rho^2)^4\hat\rho\langle \nabla_{\hat g} \hat \rho,\hat\nu\rangle_{\hat g}\geq 0.
$$
Since we have $\dim \hat M=n+1\leq 7$, it follows from geometric measure theory that there is a smooth embedded hypersurface $\hat N_{\mathrm{min},t}$ satisfying
$$\Area_{\hat g_{r,t}}(\hat N_{\mathrm{min},t})=\min_{\hat N'\in\mathcal C}\Area_{\hat g_{r,t}}(\hat N').$$
Notice that we have $\hat g_{r,t}<\hat g$ in $\hat B_r(\hat p)$ and $\hat g_{r,t}\equiv \hat g$ outside $\hat B_r(\hat p)$. By comparison and the homologically area-minimizing property of $N_{\mathrm{sys}}$ in $(M,g)$ we conclude
$$\Area_{\hat g_{r,t}}(\hat N_{\mathrm{min},t})\leq \Area_{\hat g_{r,t}}(\hat N)<\Area_{\hat g}(\hat N)=\min_{\hat N'\in\mathcal C}\Area_{\hat g}(\hat N').$$
From this we know that the hypersurface $\hat N_{\mathrm{min},t}$ has non-empty intersection with $\hat B_r(\hat p)$.

Now let us investigate the convergence of $\hat N_{\mathrm{min},t}$ as $t\to 0^+$. Recall that we have the uniform volume bound
$$\Area_{\hat g_{r,t}}(\hat N_{\mathrm{min},t})\leq \Area_{\hat g}(\hat N)= |\mathbb S^n|.$$
It follows from the curvature estimate \cite{SchoenSimonYau1975,SchoenSimon1981} for stable minimal hypersurfaces and the standard elliptic PDE theory that the hypersurfaces $\hat N_{\mathrm{min},t}$ converges to a smooth minimal hypersurface $\hat N_{\mathrm{min}}$ in $(\hat M,\hat g)$ in the $C^{5,\alpha}$-graphical sense due to the fact $\hat g_{r,t}\to \hat g$ in $C^{4,\alpha}$-sense as $t\to 0^+$. In particular, we conclude that $\hat N_{\mathrm{min}}$ encloses a region with $\hat N$ and so it satisfies
$$|\mathbb S^n|=\sys_n(M,g)\leq \Area_{\hat g}(\hat N_{\mathrm{min}})\leq |\mathbb S^n|.$$
This yields that $\hat N_{\mathrm{min}}$ is a connected, oriented, stable minimal hypersurface in $(\hat M,\hat g)$ and then it follows from Lemma \ref{Lem: infinitesimal rigidity} that $\hat N_{\mathrm{min}}$ is totally geodesic. It is also clear that $\hat N_{\mathrm{min}}$ must intersect the closure of $\hat B_r(\hat p)$ and so we have
\begin{equation}\label{Eq: distance estimate}
    \dist_{\hat g}(\hat N_{\mathrm{min}},\hat N)\leq 2r\leq 4s.
\end{equation}

We have to show that $\hat N_{\mathrm{min}}$ is different from $\hat N$ after taking $\tau$ small enough. Suppose by contradiction that $\hat N_{\mathrm{min}}$ coincides with $\hat N$. This means that $\hat N_{\mathrm{min},t}$ converge to  $\hat N$ in $C^{5,\alpha}$-graphical sense as $t\to 0^+$. Since $\hat N_{\mathrm{min},t}$ are $\hat g_{r,t}$-area-minimizing boundary, the convergence has multiplicity one and $\hat N_{\mathrm{min},t}$ must be connected.

Let $\Lambda_{\mathrm{cur}}$ be the constant from \eqref{Eq: Lambda cur} and recall that we have taken $\theta_0\in(0,1)$ and $\iota_0\in (0,1)$ such that both Lemma \ref{Lem: biRicci increase} and Lemma \ref{Lem: Ricci increase} hold. From the $C^{5,\alpha}$-graphical convergence we can guarantee for $t$ small that
\begin{itemize}
\item the second fundamental form $\hat A_t$ of $\hat N_{\mathrm{min},t}$ in $(\hat M,\hat g)$ satisfies
\begin{equation}\label{Eq: A t}
    \|\hat A_t\|_{\hat g}\leq \Lambda_{\mathrm{cur}},
\end{equation}
\item we have
\begin{equation}\label{Eq: nu t}
    \langle\nabla_{\hat g}\hat d,\hat\nu_t\rangle_{\hat g}>\sqrt{1-\iota_0^2} \mbox{ on }\hat B_r(\hat p)\cap \hat N_{\mathrm{min},t},
\end{equation}
where $\hat \nu_t$ is the unit normal vector field on $\hat N_{\mathrm{min},t}$ with respect to the metric $\hat g$ pointing to $\hat N$,
\item and we have
$$\dist_{\hat g}(\hat p,\hat N_{\mathrm{min},t})>\left(1-\frac{\theta_0}{2}\right)s.$$
\end{itemize}
Fix $\tau$ small (depending only on $\theta_0$ and $s$) such that
$$(1-\theta_0)r = (1-\theta_0)(s+\tau) \leq \left(1-\frac{\theta_0}{2}\right)s.$$
Then we have
\begin{equation}\label{Eq: in annulus}
    \hat N_{\mathrm{min},t}\cap \hat B_r(\hat p)\subset \hat{\mathcal A}_{\theta_0,r}(\hat p),
\end{equation}
where $\hat{\mathcal A}_{\theta_0,r}$ is denoted to be the annulus region $\hat B_r(\hat p)\setminus \overline{{\hat B}_{\theta_0r}(\hat p)}$.

Since $\hat N_{\mathrm{min},t}$ is the $\hat g_{r,t}$-area-minimizing boundary, it must be a two-sided stable minimal hypersurface in $(\hat M,\hat g_{r,t})$. From the stability we have
\[
\int_{\hat{N}_{min,t}}\left(\Ric^{\hat M}_{\hat{g}_{r,t}}(\hat{\nu}_{r,t},\hat{\nu}_{r,t})+\|\hat A_{r,t}\|_{\hat g_{r,t}}^{2}\right)\hat{\vp}^{2}\,\mathrm d\sigma_{\hat{g}_{r,t}}
\leq \int_{\hat{N}_{min,t}}|\hat\nabla_{\hat g_{r,t}}\hat{\vp}|_{\hat g_{r,t}}^{2}\,\mathrm d\sigma_{\hat{g}_{r,t}}
\]
for all $\hat\vp\in C_0^\infty(\hat N_{\mathrm{min},t})$, where $\hat \nu_{r,t}=e^{t(r^2-\hat\rho^2)^5}\hat \nu_t$
is the unit normal vector field and $\hat A_{r,t}$ is the corresponding second fundamental form of $\hat N_{\mathrm{min},t}$ in $(\hat M,\hat g_{r,t})$, and $\hat\nabla_{\hat g_{r,t}}$ is the covariant derivative of $\hat N_{\mathrm{min},t}$ with the induced metric from $(\hat M,\hat g_{r,t})$.   It is clear that we have
\[
\int_{\hat{N}_{min,t}}|\hat\nabla_{\hat g_{r,t}}\hat{\vp}|_{\hat g_{r,t}}^{2}\,\mathrm d\sigma_{\hat{g}_{r,t}}
= \int_{\hat N_{\mathrm{min},t}}e^{-(n-2)t(r^2-\hat\rho^2)^5}|\hat\nabla_{\hat g}\hat\vp|^2_{\hat g}\,\mathrm d\sigma_{\hat g}
\leq \int_{\hat N_{\mathrm{min},t}}|\hat\nabla_{\hat g}\hat\vp|^2_{\hat g}\,\mathrm d\sigma_{\hat g}.
\]
Recall that $\hat g_{r,t}$ coincides with the pullback metric $\phi^*g_{\phi(\hat p),r,t}$ in $\hat B_r(\hat p)$.
By the similar argument of Lemma \ref{Lem: infinitesimal rigidity} (see \eqref{Eq: spectral function-1}), we have
\[
\Ric^{\hat M}_{\hat{g}_{r,t}}(\hat{\nu}_{r,t},\hat{\nu}_{r,t})+\|\hat A_{r,t}\|_{\hat g_{r,t}}^{2}
\geq \bRic^{\hat M}_{\hat{g}_{r,t}}(\hat{\nu}_{r,t},\hat e_{r,t})-\Ric^{\hat N_{\mathrm{min},t}}_{\hat g_{r,t}}(\hat e_{r,t},\hat e_{r,t}),
\]
where $\hat e_t$ is any $\hat g$-unit tangential vector of $\hat N_{\mathrm{min},t}$ and $\hat e_{r,t}=e^{t(r^2-\hat\rho^2)^5}\hat e_t$. Then it follows from Lemma \ref{Lem: biRicci increase} and Lemma \ref{Lem: Ricci increase} as well as \eqref{Eq: A t}-\eqref{Eq: in annulus} that
\[
\begin{split}
    &\int_{\hat N_{\mathrm{min},t}}\left(\Ric^{\hat M}_{\hat{g}_{r,t}}(\hat{\nu}_{r,t},\hat{\nu}_{r,t})+\|\hat A_{r,t}\|_{\hat g_{r,t}}^{2}\right)\hat{\vp}^{2}\,\mathrm d\sigma_{\hat{g}_{r,t}}\\
    &\qquad\geq \int_{\hat N_{\mathrm{min},t}}\left(\bRic^{\hat M}_{\hat{g}_{r,t}}(\hat{\nu}_{r,t},\hat e_{r,t})-\Ric^{\hat N_{\mathrm{min},t}}_{\hat g_{r,t}}(\hat e_{r,t},\hat e_{r,t})\right)\hat{\vp}^{2}\cdot e^{-nt(r^2-\hat\rho^2)^5}\mathrm d\sigma_{\hat{g}}\\
    &\qquad\geq \int_{\hat N_{\mathrm{min},t}}\left(\bRic^{\hat M}_{\hat{g}}(\hat{\nu}_{t},\hat e_{t})-\Ric^{\hat N_{\mathrm{min},t}}_{\hat g}(\hat e_{t},\hat e_{t})+50tr^2(r^2-\hat\rho^2)^3\right)\hat{\vp}^{2}\,\mathrm d\sigma_{\hat{g}},
\end{split}
\]
So the stability inequality actually yields
\begin{equation}\label{Eq: deformed stability}
\int_{\hat N_{\mathrm{min},t}}\left(|\hat\nabla_{\hat g}\hat\vp|^2+\hat\lambda_{\Ric,t}\hat\vp^2\right)\mathrm d\sigma_{\hat g}
\geq \int_{\hat N_{\mathrm{min},t}}(n-1)+50tr^2(r^2-\hat\rho^2)^3\,\mathrm d\sigma_{\hat g},
\end{equation}
where
$$\hat\lambda_t=\hat\lambda_{\Ric,t}(x)=
 \min_{\hat e\in T_{x}\hat N_{\mathrm{min},t}, \, |\hat e|_{\hat g}=1}\Ric^{\hat N_{\mathrm{min},t}}_{\hat g}(\hat e,\hat e).
$$
From \eqref{Eq: deformed stability} we can find a smooth positive function $\hat u$ on $\hat N_{\mathrm{min},t}$ such that
$$-\Delta_{\hat g}^{\hat N_{\mathrm{min},t}}\hat u+\hat\lambda_{\Ric,t}\hat u\geq \left((n-1)+50tr^2(r^2-\hat\rho^2)^3\right)\hat u.$$
Since we have $$\Area_{\hat g}(\hat N_{\mathrm{min},t})\geq \Area_{\hat g}(\hat N)=|\mathbb S^n|,$$
it follows from Theorem \ref{Thm: spectral comparison} that $\hat N_{\mathrm{min},t}$ with the induced metric from $(\hat M,\hat g)$ is isometric to the standard unit $n$-sphere $\mathbb S^n$. This implies $\hat\lambda_{\Ric,t}\equiv (n-1)$ and so we have
$$0=-\int_{\hat N_{\mathrm{min},t}}\Delta_{\hat g}^{\hat N_{\mathrm{min},t}}\hat u\,\mathrm d\sigma_{\hat g}\geq 50\int_{\hat N_{\mathrm{min},t}}tr^2(r^2-\hat\rho^2)^3\hat u\,\mathrm d\sigma_{\hat g}.$$
In particular,  $\hat N_{\mathrm{min},t}$ does not intersect with the geodesic ball $\hat B_r(\hat p)$, which is impossible for $t$ small enough since $\hat N$ intersects with $\hat B_r(\hat p)$ and $\hat N_{\mathrm{min},t}$ converge to $\hat N$ as $t\to 0^+$.

We have constructed a totally geodesic hypersurface $\hat N_{\mathrm{min}}\neq \hat N$ in $(\hat M,\hat g)$ satisfying $\dist_{\hat g}(\hat N_{\mathrm{min},t},\hat N)\leq 4s$. By letting $s\to 0^+$ we can construct a sequence of totally geodesic hypersurfaces (on one side of $\hat N$ automatically) converging to $\hat N$, which completes the proof.
\end{proof}

\begin{lma}\label{Lem: curvature characterization}
Let $\{e_i\}_{i=1}^n$ be an orthonormal frame of $T_pN_{\mathrm{sys}}$ and $\nu$ be a unit normal vector field on $N_{\mathrm{sys}}$. Then we have $$\Rm_M(e_i,e_j,e_k,e_l)=\delta_{jk}\delta_{il}-\delta_{ik}\delta_{jl}$$ and
$$
\Rm_M(e_i,\nu,e_j,e_k)=\Rm_M(e_i,\nu,\nu,e_j)=0.
$$
\end{lma}
\begin{proof}
Since $N_{\mathrm{sys}}$ is totally geodesic and $N_{\mathrm{sys}}$ is isometric to $\mathbb{S}^n$, we have
$$\Rm_M(e_i,e_j,e_k,e_l)=\Rm_{N_{\mathrm{sys}}}(e_i,e_j,e_k,e_l)=\delta_{jk}\delta_{il}-\delta_{ik}\delta_{jl}$$
and
$$\Rm_M(e_i,\nu,e_j,e_k)=(\nabla_{e_k} A)(e_j,e_i)-(\nabla_{e_j} A)(e_k,e_i)=0.$$
It remains to show $\Rm_M(e_i,\nu,\nu,e_j)=0$.

Let us denote $N=N_{\mathrm{sys}}$ for short.
From Proposition \ref{Prop: totally geodesic approximation} we can construct a sequence of totally geodesic hypersurfaces $N_i$ on one side of $N$ converging to $N$ smoothly. After choosing a suitable unit normal vector field $\nu_N$ on $N$, we can write $N_i$ as graphs over $N$ with positive graph functions $f_i$ such that
$$f_i=o(1)\mbox{ as }i\to\infty.$$
Here and in the sequel, the notation $f_i=o(\cdot)$ means $\|f_i\|_{C^k}=o(\cdot)$ for all $k$ and the same convention is made for the notation $O(\cdot)$.
In particular, we can assume that all the hypersurfaces $N_i$ lie in a tubular $\ve$-neighborhood $\mathcal N$ of $N$ in $(M,g)$.

In the following, we take $t$ to be the smooth signed distance function to $N$ defined in $\mathcal N$ such that $N_i\subset\{t>0\}$. We consider the function
$$F_i:N\times[0,\ve)\to \mathbb R, \ \, (x,t)\mapsto t-f_i(x).$$ It is easy to see $N_i=\{F_i=0\}$ and so the minimal equation can be written as
$$\left.\Div_g\left(\frac{\nabla_g F_i}{|\nabla_gF_i|_g}\right)\right|_{F_i=0}=0,$$
or equivalently
\begin{equation}\label{Eq: minimal equation}
    \Delta_g F_i=|\nabla_gF_i|_g^{-2}\cdot\nabla_g^2F_i(\nabla_gF_i,\nabla_gF_i).
\end{equation}

We would like to compute an explicit form of this equation. For convenience, we shall denote $\nabla_t$, $\Delta_t$ and $A_t$ to be the covariant derivative, the Laplace-Beltrami operator and the second fundamental form of $N\times\{t\}$, respectively.

Since $t$ is a distance function, we can compute
$$\nabla^2_gF_i(\partial_t,\partial_t)=0, \ \,\nabla^2_gF_i(\partial_t,X)=A_t(\nabla_tf_i,X),$$
and
$$\nabla^2_gF_i(X,Y)=-\nabla_t^2 f_i(X,Y)+A_t(X,Y),$$
where $X$ and $Y$ are any vectors orthogonal to $\partial_t$. Using these expressions we can write the equation \eqref{Eq: minimal equation}  as
\begin{equation}\label{Eq: arrangement of minimal equation}
\begin{split}
&\Delta_{f_i(x)}f_i-H_{f_i(x)}=\left(1+|\nabla_{f_i(x)}f_i|^2\right)^{-1}\\
&\qquad\qquad\cdot\Big(A_{f_i(x)}(\nabla_{f_i(x)}f_i,\nabla_{f_i(x)}f_i)+\nabla_{f_i(x)}^2f_i(\nabla_{f_i(x)}f_i,\nabla_{f_i(x)}f_i)\Big).
\end{split}
\end{equation}
Using the fact $f_i=o(1)$ and the smoothness of $g$ as well as the totally-geodesic property of $N$ it is easy to verify
$$\Delta_{f_i(x)}f_i=\langle g_N+o(1),\nabla_N^2f_i\rangle+O(f_i),$$
$$H_{f_i(x)}=\int_0^{f_i(x)}\frac{\partial H_t}{\partial t}\,\mathrm dt=O(f_i),$$
$$A_{f_i}=O(f_i) \mbox{ and } \nabla_{f_i(x)}^2 f_i=\nabla_N^2 f_i+O(f_i).$$
Therefore the equation \eqref{Eq: arrangement of minimal equation} can be written as
$$\langle g_N+o(1),\nabla_N^2 f_i\rangle=O(f_i).$$
Using the Harnack inequality we see
$$\max_N f_i\leq \Lambda\min_Nf_i$$
for a universal positive constant $\Lambda$ independent of $i$. Define
$$\eta_i=\frac{f_i}{\max_Nf_i}.$$
Then $\eta_i$ is bounded from below by a positive constant and it satisfies
$$
\langle g_N+o(1),\nabla_N^2\eta_i\rangle =O(1)\mbox{ as }i\to\infty.$$
From the standard elliptic PDE theory we obtain $\eta_i=O(1)$ as $i\to\infty$. Up to a subsequence the functions $\eta_i$ converge to a smooth positive function $\eta$ on $N$ in $C^2$-sense. Using the further expansion
$$H_{f_i(x)}=-\Big(\Ric(\nu_N,\nu_N)+\|A_N\|^2\Big )f_i(x)+o(f_i(x))\mbox{ as }i\to\infty,$$
we see that $\eta$ satisfies the equation
\begin{equation*}
\Delta_N\eta+\Big(\Ric(\nu_N,\nu_N)+\|A_N\|^2\Big )\eta=0.
\end{equation*}
Recall that the normal Ricci curvature $\Ric(\nu_N,\nu_N)$ and the second fundamental form $A_N$ of $N$ both vanish. This yields that $\eta$ is a positive constant function.

Now we use the fact that each $N_i$ is totally geodesic, which yields
$$\nabla_g^2F_i=|\nabla_gF_i|^{-2}_{g}\Big(\nabla_g^2F_i(\nabla_gF_i,\cdot)\otimes\mathrm dF_i\Big).$$
Restricted to the tangential subspace orthogonal to $\partial_t$ we obtain
$$\nabla_{f_i(x)}^2 f_i-A_{f_i(x)}=\nabla_{f_i(x)}^2f_i(\nabla_{f_i(x)}f_i,\cdot)\otimes \mathrm df_i.$$
Dividing this equation by $\max_Nf_i$ and using the facts $\eta_i\to \eta\equiv \mathrm{const.}$ and $f_i=o(1)$ as well as the estimate
$$A_{f_i(x)}=\Rm_M(\cdot,\nu_N,\nu_N,\cdot)f_i(x)+o(f_i(x))\mbox{ when }A_N=0,$$
we conclude that
$$\Rm_M(X,\nu_N,\nu_N,Y)=0$$
for all vectors $X,Y$ tangential to $N$. This completes the proof.
\end{proof}

\begin{lma}\label{Lem: everywhere minimizing}
For any point $p\in M$ there is a homologically non-trivial smooth embedded hypersurface $N$ passing through $p$, which satisfies $\Area_g(N)=|\mathbb S^n|$ and the homologically area-minimizing property. In particular, the sectional curvatures of $(M,g)$ is nonnegative everywhere.
\end{lma}
\begin{proof}
As before, we consider the metric completion $(\hat M,\hat g)$ of $M\setminus N_{\mathrm{sys}}$ and denote $\hat N$ to be one of the boundary component. It suffices to show that for any point $\hat p$ there exists an embedded hypersurface $\hat N_p$ in $(\hat M,\hat g)$ satisfying $$\Area(\hat N_p)=|\mathbb S^n|,$$ which passes through the point $\hat p$ and encloses a region with $\hat N$.

To see this we consider the collection $\mathcal C$ of those hypersurfaces having the same volume as $\mathbb S^n$ and enclosing a region with $\hat N$. It is clear that $\mathcal C\neq \emptyset$ since $\hat N\in \mathcal C$. From Lemma \ref{Lem: infinitesimal rigidity} each hypersurface in $\mathcal C$ is totally geodesic and so $\mathcal C$ is compact with $C^\infty$-topology. In particular, we can find a hypersurface, denoted by $\hat N_{\hat p}$, which attain the least distance $\mathfrak d$ that hypersurfaces in $\mathcal C$ could have to the point $\hat p$.

We claim $\mathfrak d=0$, which is equivalent to say $\hat p\in\hat N_{\hat p}$. Otherwise, the point $\hat p$ lies in one component of $\hat M\setminus \hat N_{\hat p}$, denoted by $\hat\Omega_{\hat p}$. Notice that $\hat N_{\hat p}$ is one of the boundary components of $\hat\Omega_{\hat p}$. As in the proof of Proposition \ref{Prop: totally geodesic approximation} we are able to construct a sequence of hypersurfaces $\hat N_i$ in $\hat\Omega_{\hat p}$ satisfying $\Area(\hat N_i)=|\mathbb S^n|$ and enclosing a region with $\hat N_{\hat p}$. Furthermore, $\hat N_i$ converge to $\hat N_{\hat p}$ smoothly as $i\to\infty$. Then it is easy to verify  $\hat N_i\in \mathcal C$ and
$$\dist(\hat p,\hat N_i)< \dist(\hat p,\hat N_{\hat p})  \mbox{ for sufficiently large $i$}.$$
This contradicts to the definition of $\mathfrak d$.

For any point $p\in M$ let $N_p$ be the constructed homologically non-trivial hypersurface passing through $p$, which satisfies $\Area_g(N)=|\mathbb S^n|$ and the homologically area-minimizing property. It is easy to check that Proposition \ref{Prop: totally geodesic approximation} and Lemma \ref{Lem: curvature characterization}  hold for $N_p$ as well. In particular, we conclude that the sectional curvatures of $(M,g)$ is nonnegative at the point $p$.
\end{proof}

Now we are ready to prove Theorem \ref{Thm: main} under the additional dimension assumption $n+1\leq 7$.
\begin{proof}[Proof of Theorem \ref{Thm: main} when $n+1\leq 7$]
From Proposition \ref{Prop: systole inequality} we have
$$\sys_n(M,g)\leq |\mathbb S^n|.$$
So it remains to deal with the equality case.

From Lemma \ref{Lem: everywhere minimizing} we know that the Ricci curvature of $(M,g)$ is nonnegative everywhere. Based on this fact we can use the argument of Bray-Brendle-Neves \cite{BrayBrendleNeves2010} to prove that the universal cover $(\tilde M,\tilde g)$ splits into the Riemannian product $\mathbb S^n\times \mathbb R$. Here we include the details for completeness.

Denote $N_{\mathrm{sys}}$ to be the hypersurface constructed in Proposition \ref{Prop: systole inequality}. Recall that $N_{\mathrm{sys}}$ is an embedded $n$-sphere, so we can lift it to be an embedded $n$-sphere $\tilde N_{\mathrm{sys}}$ in $(\tilde M,\tilde g)$. From Proposition \ref{Prop: systole inequality} we know that $\tilde N_{\mathrm{sys}}$ is totally geodesic and that the normal Ricci curvature $\Ric_{\tilde M}(\tilde \nu,\tilde \nu)$ vanishes along $\tilde N_{\mathrm{sys}}$, where $\tilde \nu$ is a unit normal vector field on $\tilde N_{\mathrm{sys}}$. This implies that the Jacobi operator of $\tilde N_{\mathrm{sys}}$ coincides with the Laplace-Beltrami operator on $\tilde N_{\mathrm{sys}}$. From the same argument as \cite[Proposition 3.2]{BrayBrendleNeves2010} we can construct a foliation $\{\tilde N_\tau\}_{-\ve<\tau<\ve}$ with constant mean curvatures around $\tilde N_{\mathrm{sys}}$, where $\tilde N_0=\tilde N_{\mathrm{sys}}$. In particular, $\tilde \nu$ can be extended to a unit normal vector field $\tilde \nu_\tau$. Without loss of generality we assume that $\tilde\nu_\tau$ points to the same direction as $\partial_\tau$.

Since the Ricci curvature of $(M,g)$ is nonnegative, we have
 $$\int_{\tilde N_\tau}\Ric(\tilde \nu_\tau,\tilde \nu_\tau)+\|\tilde A_\tau\|^2\,\mathrm d\sigma\geq 0.$$
 Using the fact $\Ric(\tilde \nu_\tau,\tilde \nu_\tau)+\|\tilde A_\tau\|^2\to 0$ as $\tau\to 0$ and the Poincar\'e inequality we obtain for $\tau$ small enough that
 $$\int_{\tilde N_\tau}|\nabla\vp|^2-(\Ric(\tilde \nu_\tau,\tilde \nu_\tau)+\|\tilde A_\tau\|^2)\vp^2\,\mathrm d\sigma\geq 0 \, \mbox{ for all }\vp\in C^\infty(\tilde N_\tau).$$
 Denote $\rho_\tau=\langle \partial_\tau,\tilde \nu_\tau\rangle$ to be the positive lapse function. By taking
 $$\vp=\rho_\tau-\bar\rho_\tau\mbox{ with }\bar\rho_\tau=\Area(\tilde N_\tau)^{-1}\int_{\tilde N_\tau}\rho_\tau\,\mathrm d\sigma,$$
 we obtain
\begin{equation}\label{Eq: 1}
    \begin{split}
    \int_{\tilde N_\tau}|\nabla\rho_\tau|^2&-(\Ric(\tilde \nu_\tau,\tilde \nu_\tau)+\|\tilde A_\tau\|^2)\rho_\tau(\rho_\tau-2\bar\rho_\tau)\,\mathrm d\sigma\\
&\qquad\qquad\geq \bar\rho_\tau^2\int_{\tilde N_\tau}\Ric(\tilde \nu_\tau,\tilde \nu_\tau)+\|\tilde A_\tau\|^2\,\mathrm d\sigma\geq 0.
\end{split}
\end{equation}

On the other hand, one can compute
$$H'(\tau)=-\Delta\rho_\tau-(\Ric(\tilde \nu_\tau,\tilde \nu_\tau)+\|\tilde A_\tau\|^2)\rho_\tau.$$
Multiplying $\rho_\tau-2\bar\rho_\tau$ on both sides and integrating by parts we obtain from \eqref{Eq: 1} that
$$-H'(\tau)\int_{\tilde N_\tau}\rho_\tau\,\mathrm d\sigma\geq 0.$$
In particular, we obtain $H'(\tau)\leq 0$ and so $\tilde N_0$ has maximal volume among the hypersurfaces $\tilde N_\tau$. Since $\tilde N_\tau$ is homologous to $\tilde N_{\mathrm{sys}}$, we know $$\Area(\tilde N_\tau)\geq \sys_n(M,g)=|\mathbb S^n|.$$
This yields that $\Area(\tilde N_\tau)=|\mathbb S^n|$ and $\tilde N_\tau$ is area-minimizing in its homology class. From Lemma \ref{Lem: infinitesimal rigidity} we conclude that $\tilde N_\tau$ is isometric to the unit $n$-sphere $\mathbb S^n$, which is totally geodesic and has vanishing normal Ricci curvature. Then for all $\tilde N_\tau$ the Jacobi operator coincides with the Laplace-Beltrami operator and so $\rho_\tau$ are constant functions. From this it is easy to verify that a small tubular neighborhood of $\tilde N_{\mathrm{sys}}$ splits isometrically as $\mathbb S^n\times (-\ve,\ve)$. The global splitting comes from a use of the continuous method (see \cite[Proposition 3.8]{BrayBrendleNeves2010} for instance).
\end{proof}

 \section{Proof of Theorem \ref{Thm: main} using generic regularity hypothesis}\label{Sec: generic regularity}

 In this section, we don't make the dimension assumption $n+1\leq 7$ any more. For this reason, the support of mass-minimizing $n$-currents may have singularities. Our main task of this section is to show how to overcome this difficulty using the Generic Regularity
 Hypothesis.

 First we establish the systolic inequality:

  \begin{lma}\label{Lem: systolic inequality high dimension}
  Let $(M^{n+1},g)$ be an orientable $(n+1)$-dimensional closed Riemannian manifold with $H_n(M)\neq 0$ and $\bRic\geq n-1$. Suppose further that the Generic Regularity Hypothesis is true. Then we have
  $$\sys_n(M,g)\leq |\mathbb S^n|.$$
  \end{lma}

  \begin{proof}
      It suffices to prove
      \begin{equation}\label{Eq: perturb systole}
      \sys_n(M,g)\leq (1+\ve)^{2n}\cdot|\mathbb S^n|
      \end{equation} for any positive constant $\ve$. From the Generic Regularity Hypothesis for any constant $\ve>0$ we can find a smooth metric $g_\ve$ such that
      \begin{equation}\label{Eq: generic perturbation}
          \bRic_{g_\ve}\geq \frac{n-1}{(1+\ve)^2}\mbox{ and }g_\ve\geq (1+\ve)^{-2}g,
      \end{equation}
      and that for any non-trivial homology class $\alpha\in H_n(M)$ there exists a smooth, embedded, $g_\ve$-area-minimizing hypersurface $N_\alpha$ with integer multiplicity.

Fix a non-trivial homology class $\alpha\in H_n(M)$ and take the area-minimizing hypersurface $N_\alpha$ with respect to $g_\ve$. Notice that we can pick up a component of $N_\alpha$, denoted by $N_{\alpha,0}$, which is still $g_\ve$-area-minimizing and homologically non-trivial. It follows from Lemma \ref{Lem: infinitesimal rigidity} and \eqref{Eq: generic perturbation} that
$$ \Area_g(N_{\alpha,0})\leq (1+\ve)^n\cdot\Area_{g_\ve}(N_{\alpha,0})\leq (1+\ve)^{2n}\cdot|\mathbb S^n|.$$

In particular, we obtain \eqref{Eq: perturb systole}. The proof is completed by letting $\ve \to 0$.
\end{proof}

More delicate analysis needs to be done when we deal with the rigidity.
  \begin{prop}\label{Prop: existence of systole hypersurface}
      Let $(M^{n+1},g)$ be an orientable $(n+1)$-dimensional closed Riemannian manifold with $H_n(M)\neq 0$ and $\bRic\geq n-1$. Suppose that the Generic Regularity Hypothesis is true. If $\sys_n(M,g)=|\mathbb S^n|$, then  we can find a connected, smooth, embedded, homologically non-trivial hypersurface $N_{\mathrm{sys}}$ with $$\Area(N_{\mathrm{sys}})=|\mathbb S^n|.$$
  \end{prop}
  \begin{proof}
      From the proof of Lemma \ref{Lem: systolic inequality high dimension} we can find a sequence of smooth metrics $g_i$ and connected smooth embedded hypersurfaces $N_i$ with $[N_i]\neq 0\in H_n(M)$ such that
      \begin{itemize}\setlength{\itemsep}{1mm}
          \item $g_i$ converge to $g$ smoothly;
          \item $N_i$ is $g_i$-area-minimizing in its homology class;
          \item $\Area_{g_i}(N_i)=|\mathbb S^n|+o(1)$ as $i\to \infty$.
      \end{itemize}

      Fix an orientation on each hypersurface $N_i$. Since $\Area_{g_i}(N_i)$ are uniformly bounded, it follows from geometric measure theory that $N_i$ converge to a $g$-mass-minimizing $n$-current $N$ with integer multiplicity in the sense of current up to a subsequence. Moreover, the mass $\mathbb M_g(N)$ of $N$ equals to $|\mathbb S^n|$.

      Since $N$ is a $g$-mass-minimizing $n$-current, its support $\underline N$ can be decomposed as $\underline N=\mathcal R\sqcup\mathcal S$, where $\mathcal R$ is a smooth embedded open hypersurface and $\mathcal S$ is a closed subset satisfying $\mathcal H^{n-7+\ve}_g(\mathcal S)=0$ for any $\ve>0$.

      We are going to show that $\mathcal R$ is totally geodesic.
      This is based on $\mathcal L^2$-integral estimate for the second fundamental form of $N_i$. Denote $\nu_i$ and $A_i$ to be a unit normal vector field and  the corresponding second fundamental form of $N_i$ in $(M,g_i)$. Repeating the computation in the proof of Lemma \ref{Lem: infinitesimal rigidity} (see \eqref{Eq: spectral function-1}) we can derive
      \begin{equation}\label{Eq: Ricci lower bound}
          \Ric_{g_i}(\nu_i,\nu_i)\geq n-1+o(1)-\lambda_{\Ric,N_i}-\|A_i(e_i,\cdot)\|_{g_i}^2,
      \end{equation}
      where
      $$ \lambda_{\Ric,N_i}(x) = \min_{v\in T_{x}N_i, \, |v|_{g_i}=1}\Ric^{N_i}_{g_i}(v,v)$$
      and $e_i$ is a $g_i$-unit tangential vector of $N_i$ such that $$\Ric^{N_i}_{g_i}(e_i,e_i)=\lambda_{\Ric,N_i}.$$ Since $N_i$ is minimal, it is easy to derive
      \begin{equation}\label{Eq: curvature upper bound}
          \|A_i(e_i,\cdot)\|_{g_i}^2\leq \frac{n-1}{n}\|A_i\|_{g_i}^2.
      \end{equation}
    Denote $\lambda_i$ and $u_i$ to be the first eigenvalue and the first eigenfunction of the operator
    $$\mathcal J_i=-\Delta_{N_i,g_i}-\Big(\Ric_{g_i}(\nu_i,\nu_i)+\frac{n-1}{n}\|A_i\|_{g_i}^2\Big).$$
    Then we can compute from \eqref{Eq: Ricci lower bound} and \eqref{Eq: curvature upper bound} that
    $$-\Delta_{N_i,g_i}u_i+\lambda_{\Ric,N_i}u_i\geq \big(n-1+o(1)+\lambda_i\big)u_i.$$
    It follows from Theorem \ref{Thm: spectral comparison} and the assumption $\sys_n(M,g)=|\mathbb S^n|$ that
    $$|\mathbb S^n|\leq (1+o(1))\cdot \left(\frac{n-1+o(1)+\lambda_i}{n-1}\right)^{-\frac{n}{2}}|\mathbb S^n|.$$
This implies $\lambda_i=o(1)$ as $i\to\infty$. From the stability of $N_i$ we have
\[
\begin{split}
\int_{N_i}(\Ric_{g_i}(\nu_i,\nu_i)+ {} & \|A_i\|_{g_i}^2)u_i^2\,\mathrm d\sigma_{g_i}
\leq  \int_{N_i}|\nabla_{N_i,g_i}u_i|_{g_i}^2\,\mathrm d\sigma_{g_i} \\
= {} & \int_{N_i}\Big(\Ric_{g_i}(\nu_i,\nu_i)+\frac{n-1}{n}\|A_i\|_{g_i}^2-\lambda_i\Big)u_i^2\,\mathrm d\sigma_{g_i},
\end{split}
\]
which implies
$$\int_{N_i}\|A_i\|_{g_i}^2u_i^2\,\mathrm d\sigma_{g_i}\leq n\lambda_i\int_{N_i}u_i^2\,\mathrm d\sigma_{g_i}=o(1)\int_{N_i}u_i^2\,\mathrm d\sigma_{g_i}$$
and so
\[
\int_{N_i}|\nabla_{N_i,g_i}u_i|_{g_i}^2\,\mathrm d\sigma_{g_i}
= O(1)\int_{N_i}u_i^2\,\mathrm d\sigma_{g_i}.
\]

    In the following discussion, we make the normalization $$\min_{N_i}u_i=1.$$ Recall from Remark \ref{Rem: integral estimate} (with $\gamma=1$) that we actually have
    $$\int_{N_i}u_i^{\frac{2}{n-1}}\,\mathrm d\sigma_{g_i}\leq |\mathbb S^n|.$$
    Denote $\bar A_i$ to be the second fundamental form of $N_i$ in $(M,g)$. Then one can compute
    $$\|\bar A_i\|_{g}\leq \big(1+o(1)\big)\|A_i\|_{g_i}+o(1)\mbox{ as }i\to\infty.$$
    Using the Michael-Simon-Sobolev inequality on $(M,g)$ we obtain
    \begin{equation}\label{Eq: Lp control}
\begin{split}
   \left( \int_{N_i}u_i^{\frac{2n}{n-2}}\,\mathrm d\sigma_{g_i}\right)^{\frac{n-2}{n}}&=O(1)\left(\int_{N_i}u_i^{\frac{2n}{n-2}}\,\mathrm d\sigma_{g}\right)^{\frac{n-2}{n}}\\
    &\leq O(1)\left(\int_{N_i}|\nabla_{N_i,g}u_i|_g^2+\|\bar A_i\|^2_gu_i^2\,\mathrm d\sigma_{g}\right)\\
    &=O(1)\left(\int_{N_i}|\nabla_{N_i,g_i}u_i|_{g_i}^2+\|A_i\|^2_{g_i}u_i^2+u_i^2\,\mathrm d\sigma_{g_i}\right)\\
    &=O(1)\int_{N_i}u_i^2\,\mathrm d\sigma_{g_i}.
\end{split}
    \end{equation}
    It is well-known that we have the following interpolation inequality for $L^p$-spaces ($p>0$)
    \begin{equation}\label{Eq: Lp interpolation}
        \|u_i\|_{L^2(N_i,g_i)}\leq \|u_i\|^\theta_{L^{\frac{2}{n-1}}(N_i,g_i)}\cdot\|u_i\|^{1-\theta}_{L^{\frac{2n}{n-2}}(N_i,g_i)},
    \end{equation}
    where $\theta=\theta(n)\in (0,1)$ is a constant determined by
    $$\frac{1}{2}=\theta\cdot \frac{2}{n-1}+(1-\theta)\cdot \frac{2n}{n-2}.$$
    Therefore we can derive from \eqref{Eq: Lp control} and \eqref{Eq: Lp interpolation} that
    $$\|u_i\|_{L^2(N_i,g_i)}\leq O(1)\cdot\|u_i\|_{L^{\frac{2}{n-1}}(N_i,g_i)}=O(1)\mbox{ as }i\to \infty.$$
    This implies
    \begin{equation}\label{Eq: curvature estimate}
    \int_{N_i}\|A_i\|_{g_i}^2\,\mathrm d\sigma_{g_i}\leq \int_{N_i}\|A_i\|_{g_i}^2u_i^2\,\mathrm d\sigma_{g_i}=o(1)\int_{N_i}u_i^2\,\mathrm d\sigma_{g_i}=o(1)\mbox{ as }i\to\infty.
    \end{equation}

    Now we consider the convergence of $N_i$ to the regular part $\mathcal R$ of $N$. Since $N_i$ are area-minimizing, it follows from \cite[Theorem 34.5]{Simon1983} that $N_i$ converge to $N$ also in the sense of varifold. Denote $\underline {N_i}$ to be the support of $N_i$. From the monotonicity formula we can derive
    \begin{equation}\label{Eq: Hausdorff distance}
    d_{\mathcal H}(\underline {N_i},\underline N)\to 0\mbox{ as }i\to\infty.\end{equation}
    Fix a point $p\in \mathcal R$ and take a small $g$-geodesic ball $B$ centered at $p$ such that $\mathcal R\cap B$ is a connected smooth hypersurface. It follows from \cite[Corollary 27.8]{Simon1983} that $N_i|_B$ can be decomposed into a sum of $g_i$-area-minimizing boundaries. From \eqref{Eq: Hausdorff distance} we can find $g_i$-area-minimizing boundaries $\mathcal T_i$ (from the decomposition of $N_i|_B$) satisfying $\dist(p,\underline{\mathcal T_i})\to 0$ as $i\to\infty$. Using \eqref{Eq: Hausdorff distance} as well as the constancy theorem \cite[Theorem 26.27]{Simon1983} we conclude that $\mathcal T_i$ converge to $(\mathcal R\cap B,1,\xi)$ in the sense of current and varifold, where $\xi$ is an orientation on $\mathcal R\cap B$. Then it follows from the Allard regularity theorem and the standard theory of elliptic PDEs that $\underline{\mathcal T_i}$ converge to $\mathcal R\cap B$ locally in $C^2$-sense. After passing to a smaller $g$-geodesic ball $B'$ centered at $p$ it follows from \eqref{Eq: curvature estimate} that
    $$\int_{\mathcal R\cap B'}\|A\|_g^2\,\mathrm d\sigma_g=\lim_{i\to\infty}\int_{\underline{\mathcal T_i}\cap B'}\|A_i\|_{g_i}^2\,\mathrm d\sigma_{g_i}=0.$$
    In particular, we obtain $A(p)=0$.
    Let $p$ run over $\mathcal R$, and we conclude that $\mathcal R$ is totally geodesic.

    We claim $\mathcal S=\emptyset$ and so $\underline N$ is smooth. Otherwise, it follows from the standard dimension-descent argument (see \cite[Appendix A]{Simon1983} for instance) that there is a non-planar area-minimizing cone $\mathcal C^d$ in $\mathcal R^{d+1}$ with isolated singularity such that $\mathcal C^d\times \mathbb R^{n-d}$ comes from the blow-up of $\underline N$. However, the same analysis as above yields that the regular part of $\mathcal C^d$ is totally geodesic and so $\mathcal C^d$ is a hyperplane, which leads to a contradiction.

    Since we have $\mathbb M_g(N)=|\mathbb S^n|$, we conclude that $\underline N$ is connected and that $N_i$ has multiplicity one. Take $N_{\mathrm{sys}}=\underline N$. Then it is easy to verify that $N_{\mathrm{sys}}$ satisfies all our requirements.
  \end{proof}

 \begin{lma}\label{Lem: smooth neighborhood}
      Proposition \ref{Prop: totally geodesic approximation} is true.
  \end{lma}
 \begin{proof}
      We take the same construction and the same notations as in the proof of Proposition \ref{Prop: totally geodesic approximation}. To make the original argument work we need to guarantee the smoothness of $\hat N_{\mathrm{min},t}$ and $\hat N_{\mathrm{min}}$.

     Recall from Section \ref{Proposition 2.6} that in the construction of $\hat N_{\mathrm{min}}$ we fix a constant $s\in (0,\ve)$.  We claim that there exists a constant $s_*>0$  such that $\hat N_{\mathrm{min}}$ is a smooth hypersurface whenever $s<s_*$. To emphasize the dependence of $\hat N_{\mathrm{min}}$ on $s$ we shall denote it by $\hat N_{\mathrm{min}}^s$. Suppose that our claim is false. Then we can find constants $s_i$ converging to $0$ as $i\to\infty$ such that $\hat N_{\mathrm{min}}^{s_i}$ are $\hat g$-mass-minimizing boundaries with non-empty singularity. Recall from \eqref{Eq: distance estimate} that we have
       $$\dist_{\hat g}(\underline{\hat N_{\mathrm{min}}^{s_i}},\hat N)\leq 4s_i\to 0\mbox{ as }s_i\to 0.$$
       Denote $\hat N_{\mathrm{min}}^{+}$ to be the limit of $\hat N_{\mathrm{min}}^{s_i}$ in the sense of current up to a subsequence. It follows from \cite[Theorem 36.5]{Simon1983} and the strong maximum principle that $\hat N$ is a component of the support $\underline{\hat N_{\mathrm{min}}^{+}}$. Recall that we have
       $$\mathbb M_{\hat g}(\hat N_{\mathrm{min}}^+)=|\mathbb S^n|=\Area_{\hat g}(\hat N).$$
       This yields that $\hat N_{\mathrm{min}}^+=(\hat N,1,\hat \xi)$ with a suitable choice of the orientation $\hat \xi$ on $\hat N$. Since $\hat N_{\mathrm{min}}^+$ is smooth with multiplicity one, from \cite[Theorem 36.3]{Simon1983} we know that $\hat N_{\mathrm{min}}^{s_i}$ is also smooth for sufficiently large $i$. This contradicts to the assumption that $\hat N_{\mathrm{min}}^{s_i}$ has non-empty singularity.

       Since $\hat N_{\mathrm{min}}^s$ is smooth with multiplicity one for each $s<s_*$ as a $\hat g$-area-minimizing boundary, it follows from \cite[Theorem 36.3]{Simon1983} again that there is a constant $t_*=t_*(s)$ such that the $\hat g_{r,t}$-mass-minimizing current $\hat N_{\mathrm{min},t}$ is smooth as well. Moreover, $\hat N_{\mathrm{min},t}$ converge to $\hat N_{\mathrm{min}}^s$ in $C^{5,\alpha}$-graphical sense due to the Allard regularity theorem and the standard theory of elliptic PDEs.
       Now the original argument in the proof of Proposition \ref{Prop: totally geodesic approximation} can work smoothly and this completes the proof.
  \end{proof}

  \begin{proof}[Proof of Theorem \ref{Thm: main}]
This follows directly from Lemma \ref{Lem: systolic inequality high dimension}, Proposition \ref{Prop: existence of systole hypersurface} and Lemma \ref{Lem: smooth neighborhood} as well as the proof of Theorem \ref{Thm: main} when $n+1\leq 7$.
  \end{proof}

\appendix

\section{Curvature change under conformal deformation}\label{Appendix}

\begin{lma}\label{Conformal lemma}
Let $(M,g)$ be an $m$-dimensional Riemannian manifold and
$$\tilde g=e^{-2f}g \mbox{ with } f\in C^\infty(M).$$
Suppose that $(v,w)$ is an $g$-orthonormal pair of vectors and that $(\tilde v,\tilde w)$ is the $\tilde g$-unit normalization of $(v,w)$. Then we have
\begin{equation}\label{Eq: Ricci change}
\begin{split}
\Ric_{\ti{g}}(\ti{v},\ti{v})
=e^{2f}\Big( \Ric_{g}(v,v)&+\Delta_{g}f+(m-2)\nabla_{g}^{2}f(v,v)\\
&\qquad-(m-2)|\nabla_{g}f|_g^{2}+(m-2)v(f)^{2} \Big)
\end{split}
\end{equation}
and
\begin{equation}\label{Eq: biRicci change}
\begin{split}
\bRic_{\ti{g}}&(\ti{v},\ti{w})
=  e^{2f}\Big( \bRic_{g}(v,w)+2\Delta_{g}f+(m-3)\big(\nabla_{g}^{2}f(v,v)\\
&+\nabla_{g}^{2}f(w,w)\big)  -(2m-5)|\nabla_{g}f|_g^{2}+(m-3)\big(v(f)^{2}+w(f)^{2}\big)\Big).
\end{split}
\end{equation}
\end{lma}

\begin{proof}
It is well-known that we have
\[
\begin{split}
\Rm_{\ti{g}}(\ti{v},\ti{w},\ti{w},\ti{v})
= e^{2f}&\Big( \Rm_{g}(v,w,w,v)+\nabla_{g}^{2}f(v,v)\\
&\qquad+\nabla_{g}^{2}f(w,w)+v(f)^{2}+w(f)^{2}-|\nabla_{g}f|_g^{2} \Big)
\end{split}
\]
for any $g$-orthonormal pair of vectors $(v,w)$. The formulas for Ricci curvature and bi-Ricci curvature follow from suitable linear combinations of the above formula.
\end{proof}

\end{document}